\documentclass[draft]{article}

\usepackage[leqno,fleqn]{amsmath}
\usepackage{amsthm}
\usepackage{amsbsy}
\usepackage{comment}
\usepackage{enumerate}
\usepackage[normalem]{ulem}

\usepackage{tikz}
\usetikzlibrary{positioning,patterns,calc}

\usepackage{amsfonts}
\usepackage{mathrsfs}
\usepackage{MnSymbol}
\DeclareMathAlphabet{\mathpzc}{OT1}{pzc}{m}{it}

\newcommand{\duline}[1]{{\bgroup \markoverwith{{\bgroup \markoverwith{\rule[-1.2pt]{0.1pt}{0.4pt}}\ULon {\rule[-2.8pt]{1pt}{0.4pt}}}}\ULon {#1}}} 
\newcommand{\dulineF}[1]{{\bgroup \markoverwith{{\bgroup \markoverwith{\rule[-1.2pt]{0.4pt}{0.4pt}}\ULon {\rule[-2.8pt]{2pt}{0.4pt}}}}\ULon {#1}}} 
\newcommand{\dulineeta}[1]{{\bgroup \markoverwith{{\bgroup \markoverwith{\rule[0.2pt]{0.1pt}{0.4pt}}\ULon {\rule[-1.4pt]{1pt}{0.4pt}}}}\ULon {#1}}} 

\newcommand{\real}{\mathbb{R}} 
\newcommand{\rest}[1]{\,\rule[-.17cm]{.012cm}{.3cm}_{\,#1}} 
\renewcommand{\setminus}{\,\backslash\,}

\newcommand{\cptf}{\mathcal{C}^\infty_c} 

\renewcommand{\index}{\textup{Index}} 
\newcommand{\snabla}{{}\nabla^\Sigma} 
\newcommand{\rnabla}{\nabla^{\circ}} 
\newcommand{\meanvec}{\vec{H}} 
\newcommand{\meancurv}{H} 
\newcommand{\trace}{\textup{tr}} 
\newcommand{\divergence}{\textup{div}} 
\newcommand{\dvol}{\textup{dvol}} 
\newcommand{\normalvec}{\vec{N}} 
\newcommand{\gauss}{e^{-\frac{|x|^2}{4}}} 
\newcommand{\mgauss}{e^{\frac{|x|^2}{4}}}
\newcommand{\slaplace}{\Delta^{\Sigma}} 
\newcommand{\dgauss}{\textup{d}\mu} 
\newcommand{\wslaplace}{\Delta_\mu^\Sigma} 
\newcommand{\wsjacobi}{L^\Sigma_\mu} 
\newcommand{\wltwo}{{L^2_\mu}} 
\newcommand{\wltwoforms}{\wltwo\Omega} 
\newcommand{\hodged}{d^*_\mu} 
\newcommand{\ghf}{\mathcal{H}} 
\newcommand{\whodge}{\star_{\mu}} 

\theoremstyle{plain}
\newtheorem{theorem}{Theorem}[section]
\newtheorem{proposition}[theorem]{Proposition}
\newtheorem{lemma}[theorem]{Lemma}
\newtheorem{corollary}[theorem]{Corollary}
\newtheorem*{theorem*}{Theorem}
\newtheorem{innercustomthm}{Theorem}
\newenvironment{maintheorem}[1]
  {\renewcommand\theinnercustomthm{#1}\innercustomthm}
  {\endinnercustomthm}
\newtheorem{claim}{\texttt{Claim}}
\makeatletter
\@addtoreset{claim}{theorem}
\makeatother

\theoremstyle{definition}
\newtheorem{definition}[theorem]{Definition}

\theoremstyle{remark}
\newtheorem{remark}{\underline{Remark}}
\numberwithin{equation}{theorem}

\begin{document}

\title{Index Estimate of Self-Shrinkers in $\real^3$ with Asymptotically Conical Ends}
\author{Nicolau Sarquis Aiex
}
\maketitle

\renewcommand{\abstractname}{\vspace{-\baselineskip}}
\begin{abstract}
\noindent \textsc{Abstract.} We construct Gaussian Harmonic forms of finite Gaussian weighted $L^2$-norm on non-compact surfaces that detect each asymptotically conical end. As an application we prove an extension of the index estimates of self-shrinkers in \cite{mcgonagle2015} under the existence of such ends. We show that the Morse index of a self-shrinker is greater or equal to $\frac{2g+r-1}{3}$, where $r$ is the number of asymptotically conical ends.
\end{abstract}


\section{Introduction}
A surface $\Sigma$ in $\real^3$ is called a self-shrinker if each time slice of its mean curvature flow is a re-scaling of itself.
Self-shrinkers are models for the singularities of the mean curvature flow of an arbitrary surface, as shown by Huisken \cite{huisken1990} and Ilmanen-White.
It is a long standing problem to classify self-shrinkers in order to better understand the flow.

Self-shrinkers may be defined as critical points of the Gaussian weighted area in $\real^3$.
In this work we will be interested in studying its Morse index when the surfaces have conical structure at each end.
It was shown by L. Wang in \cite[Theorem 1.1]{wang_lu2014} that the conical asymptote uniquely defines the whole surface.

There are two well known conjectures \cite{ilmanen-unpublished} (see also \cite[1.2;1.3]{wang_lu2014}) on the rigidity of complete embedded non-compact self-shrinkers.

\noindent \textbf{Conjecture 1:} Each end of a self-shrinker is either asymptotically conical or asymptotically cylindrical.

\noindent \textbf{Conjecture 2:} The only self-shrinker with asymptotically cylindrical end is $S^1(\sqrt{2})\times\real$ (up to isometry).

In view of these conjectures and Wang's uniqueness theorem, our assumption of asymptotically conical ends seems to not be very restrictive.
In fact, if both conjectures are proven to be true, our estimates would be valid for all self-shrinking surfaces.

Let us mention that Q. Ding-L. Xin \cite[1.1]{ding-xin2013} proved that a self-shrinker of arbitrary codimension in any euclidean space has polynomial volume growth, so we no longer need that assumption.

Our index estimates are obtained following the ideas of M. Mcgonagle \cite{mcgonagle2015}, in which the he uses closed forms with finite Gaussian weighted norm to construct Gaussian Harmonic $1$-forms.
The coordinates of these forms are then used as test-functions on the second variation of the Gaussian weighted area of the self-shrinker.
Mcgonagle proved that $\index(\Sigma)\geq\frac{g}{3}$ where $g$ is the genus.
These ideas are similar to \cite{ros2006} and have been exploited with much success in the study of minimal surfaces.

The main result in this article is the construction of closed forms with finite Gaussian norm that detect the asymptotically conical ends.
We show that the space of Gaussian Harmonic $1$-forms has at least $2g+r-1$ linearly independent $1$-forms, where $r$ is the number of asymptotically conical ends.
More specifically:
\begin{maintheorem}{\ref{dimension}}
If $\Sigma$ is a complete properly immersed surface in $\real^3$ of genus $g$ and $r$ asymptotically conical ends, then $dim(\wltwoforms^1\cap\ghf^1)\geq 2g+r-1$.
\end{maintheorem}
That way we are able to mildly extend Mcgonagle's result so to include the number of ends as well as the genus on the index estimates.
We prove the following:
\begin{maintheorem}{\ref{indexestimate}}
Let $\Sigma$ be a properly immersed orientable self-shrinking surface in $\real^3$ with genus $g$ and $r$ asymptotically conical ends.
If the principal curvatures satisfy $|\kappa_1^2-\kappa_2^2|\leq c <1$ then, 
\begin{equation*}
\index(\Sigma)\geq\frac{2g+r-1}{3}.
\end{equation*}
\end{maintheorem}
This result may be compared to O.Chodosh-D.Maximo \cite{chodosh-maximo2016} extension of Ros' result mentioned above.

This paper is divided as follows. In section $2$ we provided the necessary preliminaires and definitions about self-shrinkers and Morse index. In section $3$ we construct closed $1$-forms that detect each asymptotically conical end and prove existence of Gaussian Harmonic Forms. In section $4$ we prove the index estimates.

We would like to remark that during the review of the final version of this work, D.Impera-M.Rimoldi-A.Savo \cite{impera-rimoldi-savo2018:arxiv} have published a major improvement to Mcgonagle's result by removing the condition on the principal curvatures as well as improving the index estimate.
Their methods are slightly different and in fact more general as they work with $f$-minimal hypersurfaces.
It is not clear what is the appropriate asymptotic structure at infinity for an arbitrary $f$-weighted area functional.
Together with their result we can marginally improve our index estimates (see section $4$).
We point out that our result remains novel since it also includes the contribution of each asymptotically conical end to the index.

\hfill

\textit{
Acknowledgements: I would like to thank Robert Haslhofer for suggesting this problem and for many helpful conversations at the beginning of this project. I am also thankful to Ailana Fraser and Jingyi Chen for several comments and discussions on this work. This project began when the author was a Postdoctoral Fellow at Universty of Toronto and I am grateful to Alex Nabutovsky for the invitation and the university for the hospitality.
}



\section{Preliminaires}
Let $\Sigma^2\subset \real^{3}$ be a smooth properly immersed surface of $\real^{3}$.
Denote by $\langle \cdot,\cdot \rangle$ the Euclidean metric, $\rnabla$ its connection, $g_\Sigma$ the induced Riemannian metric on $\Sigma$ and $\snabla$ its Levi-Civita connection.
The second fundamental form of $\Sigma$ is defined as $A_\Sigma(X,Y)=\rnabla_x Y - \snabla_X Y$ and the mean curvature vector is its trace on $\Sigma,$ $\meanvec_\Sigma=\trace_\Sigma(A_\Sigma)$.


\begin{definition}
We say that $\Sigma$ is a self-shrinker in $\real^{3}$ if
\begin{equation*}
   \meanvec_{\Sigma}(x)+\frac{x^\perp}{2}=0,
\end{equation*}
where $x^\perp$ denotes the projection onto the normal subspace $T_x\Sigma^\perp$ at $x\in \Sigma$.
\end{definition}

\begin{remark}
If $\Sigma$ is a self-shrinker then $\Sigma_t=\sqrt{-t}\Sigma$ is a solution for the mean curvature flow for $t\in(-\infty,0)$.
Vice-versa, if $\Sigma_t$ flows by mean curvature and $\Sigma_t=\sqrt{-t}\Sigma_{-1}$ then $\Sigma_{-1}$ is a self-shrinker (see \cite[2.2]{colding-minicozzi2012}).
\end{remark}

We will, however, be mostly interested in the following variational characterization of self-shrinkers due to Colding-Minicozzi \cite{colding-minicozzi2012}.

\begin{definition}
Let $\Sigma\subset\real^{3}$, $x_0\in\real^{3}$ and $t_0\geq 0$ and define the functional $F_{x_0,t_0}$ as
\begin{equation*}
   F_{x_0,t_0}(\Sigma)=\frac{1}{4\pi t_0}\int_\Sigma e^{-\frac{|x-x_0|^2}{4t_0}}\dvol_\Sigma,
\end{equation*}
where $\dvol_\Sigma$ denotes the volume form of $\Sigma$.

We are only interested when $x_0=0$ and $t_0=1$.
Henceforth denote the $F$-functional as $F=F_{0,1}$.
\end{definition}

Critical points of the $F$-functional are in fact minimal surfaces with respect to the incomplete Gaussian metric in $\real^{3}$.
%
%
Throughout the remainder of this section let $\Phi_s:\Sigma\rightarrow\real^{3}$ be a one-parameter family of embeddings of $\Sigma$ with $\Phi_0=id_\Sigma$ the identity map.
Denote $\Sigma_s=\Phi_s(\Sigma)$ and $X_s(x)=\frac{\partial\Phi_s(x)}{\partial s}$ the variation vector field.

\begin{proposition}[First and Second variation formulas, {\cite[3.1, 4.14]{colding-minicozzi2012}}]
Let $\Sigma_s$ be a one-parameter variation of $\Sigma$ with variation vector field $X_s$.
Suppose $X_s$ has compact support and 
$X_0=u\cdot\normalvec_{\Sigma}$ where $u:\Sigma\rightarrow\real$ is a compactly supported smooth function on $\Sigma$.
Then,
\begin{equation*}
   \frac{d}{ds}\rest{s=0}F(\Sigma_s)=\frac{1}{4\pi}\int_\Sigma\left[\left(\meancurv_\Sigma-\langle\frac{x}{2},\normalvec_\Sigma\rangle\right)u\right]\gauss \dvol_\Sigma.
\end{equation*}
Consequently $\Sigma$ is a self-shrinker if, and only if, it is a critical point for the $F$-functional.

If in addition $\Sigma$ is a self-shrinker and no boundary then,
\begin{equation*}
   \begin{aligned}
      \frac{d^2}{ds^2}\rest{s=0}F(\Sigma_s) & = \frac{1}{4\pi}\int_\Sigma u\left[-\slaplace u +\langle\frac{x}{2},\snabla u\rangle-\left(|A_\Sigma|^2+\frac{1}{2}\right)u\right]\gauss \dvol_\Sigma\\
                                            & = \frac{1}{4\pi}\int_\Sigma\left[|\snabla u|^2-\left(|A_\Sigma|^2+\frac{1}{2}\right)u^2\right]\gauss\dvol_\Sigma,
   \end{aligned}
\end{equation*}
where $\slaplace u=\divergence_\Sigma(\snabla u)$ is the Laplacian on $\Sigma$ and $|A_\Sigma|^2=\trace(A_\Sigma\circ A_\Sigma^{T})$ is the squared norm of the second fundamental form.
\end{proposition}



To simplify notation, let us write $\dgauss_\Sigma=\gauss\dvol_\Sigma$ for the weighted volume form, $\wslaplace u = \slaplace u - \langle\frac{x}{2},\snabla u\rangle$ for the weighted Laplacian on $\Sigma$, $\wsjacobi=\wslaplace +\left(|A_\Sigma|^2+\frac{1}{2}\right)$ for the weighted Jacobi operator on $\Sigma$ and $Q^\Sigma(u,u)=-\int_\Sigma u\wsjacobi u\dgauss$ its associated quadratic form.

\begin{definition}
For each $R>0$ we define $\index(\Sigma\cap B_R(0))$ as the number of negative eigenvalues of $\wsjacobi$ in $\Sigma\cap B_R(0)$ with Dirichlet boundary conditions. That is, solutions to
\begin{equation*}
   \left \{
      \begin{aligned}
      & \wsjacobi u +\lambda u = 0, \text{ on } \Sigma\cap B_R(0)\\
      & u=0 \text{, on } \partial(\Sigma\cap B_R(0))
      \end{aligned}
   \right.
\end{equation*}
with $\lambda<0$.
The Morse index of a self-shrinker $\Sigma$ is defined as $\index(\Sigma)=\limsup_{R\rightarrow\infty}\index(\Sigma\cap B_R(0))$, possibly being infinite.
\end{definition}

\begin{remark}
It is important to mention that the definition of $F$-stability in \cite{colding-minicozzi2012} is different from having positive Morse index as we defined above.
In their definition they consider the functional $F_{x,t}$ and also take into account variations of $x$ and $t$, whereas we have fixed $x=0$, $t=1$ and are looking at the Morse index of $F_{0,1}$ at a critical point.
In fact, it follows from \cite[5.2,9.2]{colding-minicozzi2012} that every self-shrinker has Morse index at least $1$.
\end{remark}

Let us define $\wltwo(\Sigma)$ as the completion of $\cptf(\Sigma)$ under the $\wltwo$-norm.
We conclude this section with a result due to D. Fischer-Colbrie \cite{fischer-colbrie1985}.
The proof in \cite{fischer-colbrie1985} carry out trivially in the weighted $L^2$-space.
See also \cite[1.4]{devyver2012} for a more general result.

\begin{proposition}[{\cite[p.124]{fischer-colbrie1985}}]\label{fischercolbrie}
A self-shrinker $\Sigma$ has finite index $\index(\Sigma)=k$ if, and only if, there exists and $\wltwo$-orthonormal set of eigenfunctions $f_1,\ldots,f_k\in \wltwo(\Sigma)$ with eigenvalues $\lambda_1,\ldots,\lambda_k<0$ and every $\phi\in\cptf(\Sigma)$ $\wltwo$-orthogonal to $f_i$ for all $i$ satisfies $Q^\Sigma(\phi,\phi)\geq 0$.
\end{proposition}

\section{Asymptotically Conical ends}
In this section we will use the asymptotic structure to construct $\wltwo$-integrable closed $1$-forms.

\begin{definition}
A surface $M\subset\real^{3}$ is said to be asymptotically conical if there exists a cone $C_\gamma=\{r\cdot\gamma: r>0\}$, where $\gamma\subset S^2$ is a smooth closed curve, such that $\lambda M$ converges graphically smoothly in compact sets to $C_\gamma$ as $\lambda\rightarrow 0^+$.
That is, for all $R>0$, $\lambda M\cap B_R\setminus\bar{B}_{\frac{1}{R}}$ converges graphically smoothly to $C_\gamma\cap B_R\setminus\bar{B}_{\frac{1}{R}} $, with $B_R$ being the Euclidean ball centered at $0$ of radius $R$.

We say that $\Sigma\subset \real^3$ has $r$ asymptotically conical ends if there exists $R_0>0$ such that $\Sigma\setminus \bar{B}_{R_0}=E_1\cup\ldots\cup E_r$ where $\{E_i\}_{i=1,\ldots,r}$ are pairwise disjoint asymptotically conical.
\end{definition}

For an asymptotically conical surface we have the following coordinate system proved in \cite{wang_lu2014}.

\begin{lemma}[{\cite[2.2]{wang_lu2014}}]\label{ends coordinate}
If $\Sigma$ is asymptotic to a cone $C_\gamma$ then there exist constants $R_1, C>0$ depending only on the curvature of $\gamma$ in $S^2$, $K\subset \Sigma$ a compact set and $v:C_\gamma\setminus B_{R_1}\rightarrow\real$ a smooth function such that:
\begin{enumerate}[(i)]
\item $\|v\|_{C^1}\leq C$ and
\item $\Sigma\setminus K=\{x+v(x)\vec{N}_{C_\gamma}(x): x\in C_\gamma\setminus B_{R_1}\}$,
\end{enumerate}
where $\vec{N}_{C_\gamma}$ is the normal vector field on the cone.
\end{lemma}

In \cite{wang_lu2014} Wang proves a more general result for self-shrinkers. 
However, the proof of the above statement only requires $\lambda\Sigma$ to converge to a cone as $\lambda\rightarrow 0^+$.

\subsection*{$\wltwo$-Integrable Closed $1$-forms}
The following construction will be the basis to find non-trivial Gaussian Harmonic Forms for each pair of asymptotically conical ends.

\begin{lemma}\label{closed forms construction}
Let $\Sigma\subset\real^{3}$ be a surface with asymptotically conical ends and denote by $\bar{\Sigma}$ its one-point compactification.
If $\beta:[0,1]\rightarrow\Sigma$ is a separating simple closed curve in $\Sigma$ but non-separating in $\bar{\Sigma}$ then there exists $\alpha:\real\rightarrow\Sigma$ that intersects $\beta$ at exactly one point and a $\wltwo$-integrable closed $1$-form $\eta$ supported in a neighbourhood of $\alpha$ such that $\int_\beta\eta=1$.

Furthermore, for any simple closed curve $\sigma$ that intersects $\alpha$ only at one point transversally, there exists a closed $1$-form $\tau$ supported at a neighbourhood of $\sigma$ such that $\int_\Sigma\tau\wedge\eta=\int_\sigma\eta$.

\end{lemma}
\begin{proof}
Firstly, observe that the existence of such a curve $\beta$ implies that $\Sigma$ must have at least two ends.
So, there is $R_0>0$ such that $\Sigma\setminus B_{R_0}=E_1\cup\ldots\cup E_r$ where each $E_i$ is asymptotic to a cone $C_{\gamma_i}$ and $r\geq 2$.

Secondly, because $\beta$ is separating in $\Sigma$ we can write $\Sigma\setminus\beta=U_1\cup U_2$, where $U_i$ are disjoint open sets, $i=1,2$.
By abuse of notation we identify $\beta$ with its image.
Since $\beta$ is non-separating in $\bar\Sigma$ we may infer that each $U_i$ must contain at least one end.
Without loss of generality we assume that $E_1\subset U_1$ and $E_2\subset U_2$.

Let $R_1>0$, $K_i\subset E_i$ and $v_i:C_{\gamma_i}\setminus R_1\rightarrow \real$ be as in Lemma \ref{ends coordinate}.
For each $i=1,2$ we have the following coordinate system $\Phi_i:[R_1,\infty)\times[0,1]\rightarrow E_i\setminus K_i$ given by
\begin{equation*}
   \Phi_i(r,s)=r\gamma_i(s)+v_i(r\gamma_i(s))\vec{N}_{\gamma_i}(s),
\end{equation*}
where $\gamma_i$ is parametrized in $[0,1]$, $\vec{N}_{\gamma_i}(s)$ its normal vector in $S^2$, which is equal to $\vec N_{C_{\gamma_i}}(r\gamma_i(s))$.

Let $\alpha_i(t)=\Phi_i(t,\frac{1}{2})$ and $\tilde\alpha:[-R_1,R_1]\rightarrow \Sigma\setminus(E_1\setminus K_1\cup E_2\setminus K_2)$ be any smooth curve connecting $\alpha_1(R_1)$ to $\alpha_2(R_1)$ and intersecting $\beta$ at only one point and transversally.
We may assume that $\tilde\alpha$ intersects $\beta$ at $\beta(\frac{1}{2})$.

Finally, define $\alpha:\real\rightarrow\Sigma$ as the concatenation of the three curves $\alpha_1$, $\tilde\alpha$ and $\alpha_2$.
In particular, $\alpha$ intersects $\beta$ at only one point and $\alpha(t)=\alpha_1(t)$ for $t\geq R_1$, $\alpha(t)=\alpha_2(-t)$ for $t\leq -R_1$ and $\alpha(t)=\tilde\alpha(t)$ for $t\in[-R_1,R_1]$.

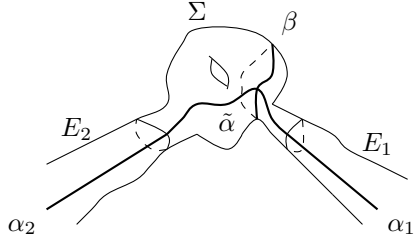
\begin{center}
\begin{tikzpicture}[scale=0.4]
\draw  plot[smooth, tension=.7] coordinates {(-0.5,-0.5) (0.5,-1) (1,-0.5)};
\draw  plot[smooth, tension=.7] coordinates {(1,-0.5) (1.5,0) (2,-0.5) (2.5,-1) (3,-1.5) (3.5,-2) (4,-2.5)};
\draw  plot[smooth, tension=.7] coordinates {(2,0.5) (3,0) (4,-0.5) (4.5,-1) (5.5,-1.5) (6.5,-2)};
\draw  plot[smooth, tension=.7] coordinates {(4,-2.5) (4.5,-3) (5,-3.5)};
\draw  plot[smooth, tension=.7] coordinates {(2,0.5) (2.5,1.5) (2,2.5) (1,3) (-0.5,3) (-1,2.5) (-1.5,1.5) (-1.5,0.5) (-2.5,0) (-3.5,-0.5) (-4.5,-1) (-5.5,-1.5)};
\draw  plot[smooth, tension=.7] coordinates {(-0.5,-0.5) (-1.5,-1) (-2,-1.5) (-2.5,-2) (-3.5,-2.5) (-4,-3) (-4.5,-3.5)};
\draw  plot[smooth, tension=.7] coordinates {(-0.1,2.1) (0.5,1.5) (0.5,1)};
\draw  plot[smooth, tension=.7] coordinates {(0,2) (0,1.5) (0.5,1.1)};
\draw  plot[smooth, tension=.7] coordinates {(2,2.5)};
\draw[thick]  plot[smooth, tension=.7] coordinates {(2,2.5) (2,1.5) (1.5,1) (1.5,0)};
\draw[dashed]  plot[smooth, tension=.7] coordinates {(1.5,0) (1,0.5) (1,1) (1,1.5) (1.5,2) (2,2.5)};
\node[above right] at (2, 2.5) {$\beta$};
\draw  plot[smooth, tension=.7] coordinates {(-2.5,0) (-1.5,-0.5) (-1.5,-1)};
\draw  plot[smooth, tension=.7] coordinates {(-1.5,-1)};
\draw[dashed]  plot[smooth, tension=.7] coordinates {(-1.5,-1) (-2,-1) (-2.5,-0.5) (-2.5,0)};
\draw  plot[smooth, tension=.7] coordinates {(2.5,-1) (2.5,-0.5) (3,0)};
\draw  plot[smooth, tension=.7] coordinates {(3,0)};
\draw[dashed]  plot[smooth, tension=.7] coordinates {(3,0) (3,-1) (2.5,-1)};
\draw[thick] (-1.5,-0.5) node (v1) {} -- (-5.5,-3) node[below left] {$\alpha_2$}  ;
\draw[thick] (2.5,-0.5) node (v2) {} -- (5.5,-3) node [below right] {$\alpha_1$};
\draw[thick]  plot[smooth, tension=.7] coordinates {(v1) (-1,0) (-0.5,0.5) (0.5,0.5) (1.5,1) (2,0) (v2)};
\node[below] at (0.5,0.5) {$\tilde\alpha$};
\node[above] at (5.5,-1.5) {$E_1$};
\node[above] at (-4.5,-1) {$E_2$};
\node[above] at (-0.5,3) {$\Sigma$};
\end{tikzpicture}
\end{center}

Now, by using parallel transport, for some $\varepsilon_1>0$ sufficiently small, we can define a smooth tubular neighbourhood $\tilde\Phi:[-R_1,R_1]\times(\frac{1}{2}-\varepsilon_1,\frac{1}{2}+\varepsilon_1)\rightarrow \Sigma$ of $\tilde\alpha$ so that:
\begin{enumerate}[(a)]
\item $\tilde\Phi(r,\frac{1}{2})=\tilde\alpha(r)$
\item $\tilde\Phi(0,s)=\beta(s)$
\item $\tilde\Phi(R_1,s)=\Phi_1(R_1,s)$ and $\tilde\Phi(-R_1,s)=\Phi_2(R_1,s)$
\end{enumerate}

We then define an open neighbourhood $\Phi:\real\times(\frac{1}{2}-\varepsilon_1,\frac{1}{2}+\varepsilon_1)\rightarrow\Sigma$ of $\alpha$ as follows:
\begin{equation*}
   \Phi(r,s)=\left\{
   \begin{aligned}
      & \Phi_1(r,s), \text{ if } r\geq R_1\\
      & \tilde\Phi(r,s), \text{ if } r\in(-R_1,R_1)\\
      & \Phi_2(-r,s), \text{ if } r\leq -R_1.
   \end{aligned}
   \right.
\end{equation*}

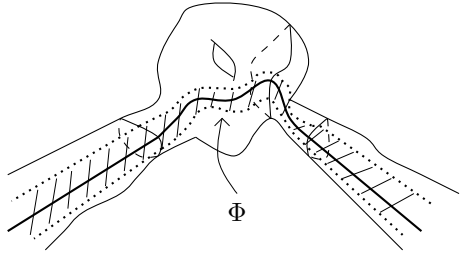
\begin{center}
\begin{tikzpicture}[scale=0.5]
\draw  plot[smooth, tension=.7] coordinates {(-0.5,-0.5) (0.5,-1) (1,-0.5)};
\draw  plot[smooth, tension=.7] coordinates {(1,-0.5) (1.5,0) (2,-0.5) (2.5,-1) (3,-1.5) (3.5,-2) (4,-2.5)};
\draw  plot[smooth, tension=.7] coordinates {(2,0.5) (3,0) (4,-0.5) (4.5,-1) (5.5,-1.5) (6.5,-2)};
\draw  plot[smooth, tension=.7] coordinates {(4,-2.5) (4.5,-3) (5,-3.5)};
\draw  plot[smooth, tension=.7] coordinates {(2,0.5) (2.5,1.5) (2,2.5) (1,3) (-0.5,3) (-1,2.5) (-1.5,1.5) (-1.5,0.5) (-2.5,0) (-3.5,-0.5) (-4.5,-1) (-5.5,-1.5)};
\draw  plot[smooth, tension=.7] coordinates {(-0.5,-0.5) (-1.5,-1) (-2,-1.5) (-2.5,-2) (-3.5,-2.5) (-4,-3) (-4.5,-3.5)};
\draw  plot[smooth, tension=.7] coordinates {(-0.1,2.1) (0.5,1.5) (0.5,1)};
\draw  plot[smooth, tension=.7] coordinates {(0,2) (0,1.5) (0.5,1.1)};
\draw  plot[smooth, tension=.7] coordinates {(2,2.5)};
\draw  plot[smooth, tension=.7] coordinates {(2,2.5) (2,1.5) (1.5,1) (1.5,0)};
\draw[dashed]  plot[smooth, tension=.7] coordinates {(1.5,0) (1,0.5) (1,1) (1,1.5) (1.5,2) (2,2.5)};
\draw  plot[smooth, tension=.7] coordinates {(-2.5,0) (-1.5,-0.5) (-1.5,-1)};
\draw  plot[smooth, tension=.7] coordinates {(-1.5,-1)};
\draw[dashed]  plot[smooth, tension=.7] coordinates {(-1.5,-1) (-2,-1) (-2.5,-0.5) (-2.5,0)};
\draw  plot[smooth, tension=.7] coordinates {(2.5,-1) (2.5,-0.5) (3,0)};
\draw  plot[smooth, tension=.7] coordinates {(3,0)};
\draw[dashed]  plot[smooth, tension=.7] coordinates {(3,0) (3,-1) (2.5,-1)};
\draw[thick] (-1.5,-0.5) node (v1) {} -- (-5.5,-3);
\draw[thick] (2.5,-0.5) node (v2) {} -- (5.5,-3);
\draw[thick]  plot[smooth, tension=.7] coordinates {(v1) (-1,0) (-0.5,0.5) (0.5,0.5) (1.5,1) (2,0) (v2)};



\draw[dotted, thick] (-1.4,-0.8) node (v4) {} -- (-4.8,-3.2);
\draw[dotted, thick] (-2,-0.2) node (v3) {} -- (-5.4,-2.2);
\draw[dotted, thick]  plot[smooth, tension=.7] coordinates {(v3) (-1.6,0) (-1.2,0.2) (-1,0.4) (-0.8,0.6) (-0.4,0.8) (0.2,0.8) (0.6,0.8) (1,1) (1.2,1.2) (1.6,1.2) (1.8,1) (2,0.8) (2,0.6) (2,0.2) (2.2,0) (2.6,-0.2)};
\draw[dotted, thick]  plot[smooth, tension=.7] coordinates {(2.6,-0.2) (2.8,-0.2)};
\draw[dotted, thick]  plot[smooth, tension=.7] coordinates {(v4) (-1.1,-0.6) (-0.7,-0.2) (-0.4,0.2) (0,0.3) (0.3,0.2) (0.8,0.2) (1.1,0.4) (1.3,0.5) (1.5,0.5) (1.7,0.2) (1.9,-0.1) (2.1,-0.4) (2.3,-0.6) (2.5,-0.8)};
\draw  plot[smooth, tension=.7] coordinates {(2.8,-0.2)};
\draw[dotted, thick] (2.8,-0.2) -- (5.8,-2.3);
\draw[dotted, thick] (2.5,-0.8) -- (5,-3.1);

\draw[very thin] (-4.7,-1.9) -- (-4.9,-3.1);
\draw[very thin] (-4.2,-1.6) -- (-4.4,-2.8);
\draw[very thin] (-3.8,-1.4) -- (-3.9,-2.4);
\draw[very thin] (-3.3,-1.1) -- (-3.4,-2.1);
\draw[very thin] (-2.8,-0.8) -- (-2.9,-1.8);
\draw[very thin] (-2.3,-0.5) -- (-2.4,-1.4);
\draw[very thin] (-1.8,-0.2) -- (-1.9,-1.1);
\draw[very thin] (-1.3,0.1) -- (-1.4,-0.8);
\draw[very thin] (-0.9,0.4) -- (-1,-0.5);
\draw[very thin] (-0.4,0.8) -- (-0.5,0.1);
\draw[very thin] (0,0.7) -- (-0.1,0.3);
\draw[very thin] (0.5,0.7) -- (0.4,0.2);
\draw[very thin] (1.1,1) -- (0.9,0.3);
\draw[very thin] (1.8,1) -- (1.6,0.5);
\draw[very thin] (2,0.5) -- (1.8,0.1);
\draw[very thin] (2.4,-0.2) -- (2.1,-0.3);
\draw[very thin] (2.9,-0.3) -- (2.4,-0.7);
\draw[very thin] (3.3,-0.6) -- (2.7,-0.9);
\draw[very thin] (3.7,-0.9) -- (3,-1.2);
\draw[very thin] (4.2,-1.3) -- (3.4,-1.6);
\draw[very thin] (4.7,-1.6) -- (3.9,-2);
\draw[very thin] (5.3,-2) -- (4.4,-2.5);

\draw[very thin]  plot[smooth, tension=.7] coordinates {(0.2,-0.) (0.1,-1.1) (0.6,-2)};
\node[below] at (0.6,-2) {$\Phi$}; 
\draw[very thin] (0.4,-0.2) -- (0.2,0) -- (0,-0.2);
\end{tikzpicture}
\end{center}

Observe that the map $\Phi$ is not necessarily smooth because the curve $\alpha$ may have a kink at $\pm R_1$.
However, it follows from $(a)$-$(c)$ and the definition that $\Phi$ satisfy the following properties:
\begin{enumerate}[(i)]
\item $\Phi(r,\frac{1}{2})=\alpha(r)$,
\item $\Phi(0,s)=\beta(s)$, $\Phi(R_1,s)=\Phi_1(R_1,s)$ and $\Phi(-R_1,s)=\Phi_2(R_1,s)$,
\item $\Phi$ is smooth on $\real\times(\frac{1}{2}-\varepsilon_1,\frac{1}{2}+\varepsilon_1)\setminus\{-R_1,R_1\}\times(\frac{1}{2}-\varepsilon_1,\frac{1}{2}+\varepsilon_1)$ and
\item for any $r\in\real$, $s\mapsto \Phi(r,s)$ is smooth in $s\in(\frac{1}{2}-\varepsilon_1,\frac{1}{2}+\varepsilon_1)$.
\end{enumerate}

Let us denote $V=\Phi(\real\times(\frac{1}{2}-\varepsilon_1,\frac{1}{2}+\varepsilon_1))$ the tubular neighbourhood of $\alpha$ on $\Sigma$.
We will now work on this coordinate system near $\alpha$.
Pick $\varepsilon_2>0$ be sufficiently small so that $3\varepsilon_2<\varepsilon_1$ and let $\varphi_0:[\frac{1}{2},\frac{1}{2}+\varepsilon_1)\rightarrow[0,1]$ be a cut-off function so that $\varphi_0\equiv 1$ on $[\frac{1}{2},\frac{1}{2}+\varepsilon_2]$, $\varphi_0\equiv 0$ on $[\frac{1}{2}+2\varepsilon_2,\frac{1}{2}+\varepsilon_1)$ and $\varphi_0'(s)\leq \frac{2}{\varepsilon_2}$.
Define $\varphi:V\rightarrow \real$ as:
\begin{equation*}
   \varphi(\Phi(r,s))=\left\{
   \begin{aligned}
      & \varphi_0(s), \text{ if } s\in\left[\frac{1}{2},\frac{1}{2}+\varepsilon_1\right)\\
      & 0, \text{ otherwise.}
   \end{aligned}
   \right.
\end{equation*}

Since $\varphi$ only depends on the coordinate $s$ and $\Phi$ is smooth on $s$ then $\varphi$ is a smooth function on $V\setminus\alpha$.
Clearly $\varphi$ may be extended to $\Sigma$ and $\eta=d\varphi$ is a smooth closed $1$-form on $\Sigma\setminus\alpha$.
We observe that $\eta\equiv 0$ on a $\varepsilon_2$-neighbourhood of $\alpha$ so, in fact, $\eta$ defines a smooth closed $1$-form on $\Sigma$.
One can easily see that the support of $\eta$ is contained in $V$ and $\int_\beta\eta=1$.

Now, let $\sigma$ be any closed curve intersecting $\alpha$ at only one point and tranversally, and take a tubular neighbourhood of $\sigma$ so its transversal section coincides with $r\mapsto\Phi(r,s)$ where it intersects with $V$.
We may repeat the constructions above and define $\bar\varphi(r)$ on a neighbourhood of $\sigma$ and $\tau=d\bar\varphi$ so that $\int_\alpha\tau=1$.
Because of the transversality assumption we will have that $\tau\wedge\eta=\varphi'(s)\bar\varphi(r)dr\wedge ds$ on the intersection with $V$.
Therefore, $\int_\Sigma \tau\wedge\eta=\int_\alpha\tau\int_\sigma\eta=\int_\sigma\eta$.

It remains to show that $|\eta|^2$ is integrable with respect to the Gaussian measure.
To prove this we only need to control the growth of $|\eta|^2$ in the $\Phi_1$ and $\Phi_2$ compoments of the tubular neighbourhood of $\alpha$.

It follows from Lemma \ref{ends coordinate} that there exists a constant $\tilde C$ depending only on $\gamma_i$ such that, using the coordinate $\Phi_i$ on $E_i\setminus K_i$, we have
\begin{equation*}
   \begin{aligned}
                          1 \leq  \langle\frac{\partial\Phi_i}{\partial r}, & \frac{\partial\Phi_i}{\partial r}\rangle   \leq 1+\tilde C, \\
        r^2|\dot\gamma_i|^2 \leq  \langle\frac{\partial\Phi_i}{\partial s}, & \frac{\partial\Phi_i}{\partial s}\rangle   \leq r^2|\dot\gamma_i|^2(1+\tilde C) \text{ and }\\
                            \left|\langle\frac{\partial\Phi_i}{\partial r}\right., & \left.\frac{\partial\Phi_i}{\partial s}\rangle\right| \leq r|\dot\gamma_i|\tilde C.
   \end{aligned}
\end{equation*}

In particular, we have $r|\dot\gamma_i|\leq\sqrt{det(g_{E_i\setminus K_i})}\leq r|\dot\gamma_i|(1+\tilde C)$.
Whenever $\eta$ is not zero we have $\eta(\Phi_i(r,s))=\varphi_0'(s)ds$, so $|\eta|_{E_i\setminus K_i}^2\leq \frac{\bar C}{r^2|\dot\gamma_i|^2}$, where $\bar C>0$ only depends on $\varepsilon_2$.

Finally we compute
\begin{equation*}
   \int_{E_i\setminus K_i}|\eta|^2\dgauss\leq (1+\tilde C)\bar C\int_0^1\int_{R_1}^\infty\frac{e^{-\frac{|\Phi_i(r,s)|^2}{4}}}{r|\dot\gamma_i|}dr ds\leq \hat C\int_{R_1}^\infty \frac{e^{-\frac{r^2}{4}}}{r}dr<\infty,
\end{equation*}
where $\hat C>0$ only depends on the curve $\gamma_i$ that defines the cone and $\varepsilon_2$.

This concludes the proof.
\end{proof}

\begin{remark}
Notice that the final calculation in the proof can also be adapted to show that $|\eta|\in L^{2+\varepsilon}(\Sigma)$ for all $\epsilon>0$ and $|\eta|\nin L^{2}(\Sigma)$.
\end{remark}

\begin{remark}\label{remark on closed forms}
A similar result may be obtained when $\beta$ is a simple closed non-separating curve in $\Sigma$, in which case the corresponding curve $\alpha$ can also be taken to be closed.
In this situation there is no integrability to be proven because the $1$-form obtained would have compact support.
Furthermore, if we pick $\sigma$ to be a line instead of a closed curve, the $1$-form $\tau$ can be constructed similarly to the above so that $|\tau|$ is $\wltwo$-integrable.
\end{remark}

Let $\wltwoforms^*(\Sigma)$ be the completion of smooth differential forms $\omega$ such that $|\omega|$ is $\wltwo$-integrable under the weighted norm.
In $\wltwoforms^*$ we can define the differential operator $d$ as usual and $\hodged$ its dual with respect to the weighted inner product $\langle\cdot,\cdot\rangle_\wltwo$.
That is, $\hodged=\mgauss d^*\gauss =d^*-\frac{1}{2}\iota_{x^\top}$, where $d^*$ is the dual to $d$ with respect to the Riemannian metric on $\Sigma$.
We also define $\whodge = \mgauss\star$, where $\star$ is the Hodge operator with respect to the Riemannian metric on $\Sigma$ (see \cite[\textsection 3]{bueler1999}).

\begin{definition}
We say that $\omega$ is a Gaussian Harmonic Forms (GHF for short) if $(d\hodged +\hodged d)\omega=0$.
Denote by $\ghf^1(\Sigma)$ the space of smooth Gaussian Harmonic $1$-forms in $\Sigma$
\end{definition}


The following existence of GHF of degree $1$ is a slight generalization of the result in \cite{mcgonagle2015} by Mcgonagle, in which we also include the $1$-forms constructed above and its weighted Hodge dual.
The proof is similar but we include it here for completeness.

\begin{theorem}\label{dimension}
If $\Sigma$ is a complete properly immersed surface in $\real^3$ of genus $g$ and $r$ asymptotically conical ends, then $dim(\wltwoforms^1\cap\ghf^1)\geq 2g+r-1$.
\end{theorem}
\begin{proof}
Since $\Sigma$ has genus $g$ we may find $2g$ non-separating simple closed curves $\delta_1,\ldots,\delta_g$ and $\sigma_1,\ldots,\sigma_g$ such that each $\delta_i$ is pairwise disjoint, each $\sigma_i$ is pairwise disjoint and $\delta_i$ only intersects $\sigma_j$ when $i=j$ and the intersection is orthogonal.
We may now follow the construction of Lemma \ref{closed forms construction} and obtain smooth $1$-forms $\nu_i$ supported on a small neighbourhood of $\delta_i$ and $\tau_i$ supported on a small neighbourhood of $\sigma_i$, $i=1,\ldots,g$ satisfying:
\begin{itemize}
\item $d\nu_i=d\tau_i=0$, for all $i=1,\ldots,g$ and
\item $\int_{\Sigma}\tau_i\wedge\nu_j=\int_{\sigma_i}\nu_j=\delta_{ij}$ for all $i=1,\ldots,g$.
\end{itemize}

Now, for each asymptotically conical end $E_k$ we pick closed separating curves $\beta_k$, which we can assume are disjoint from the above constructed curves and we fix one of the ends $E_r$.
For each pair of ends $E_r,E_k$, $k=1,\ldots,r-1$ we apply the construction of Lemma \ref{closed forms construction} to obtain pairwise disjoint non-compact curves $\alpha_k$ that intersect $\beta_k$ othogonally and smooth $1$-forms $\eta_k$ supported on a neighbourhood of $\alpha_k$, $\gamma_k$ supported on a neighbourhood of $\beta_k$ for $k=1,\ldots,r-1$ such that:
\begin{itemize}
\item $\eta_k\in\wltwoforms^1(\Sigma)$ for all $k=1,\ldots,r-1$;
\item $d\eta_k=d\gamma_k=0$, for all $k=1,\ldots,r-1$ and
\item $\int_{\Sigma}\gamma_m\wedge\eta_k=\int_{\beta_m}\eta_k=\delta_{mk}$ for all $k=1,\ldots,r-1$.
\end{itemize}
Furthermore, since the $\sigma_i$ are pairwise disjoint and non-separating, we may assume that $\alpha_k\cap\sigma_i=\emptyset$ for all $i=1,\ldots,g$ and $k=1,\ldots,r-1$.

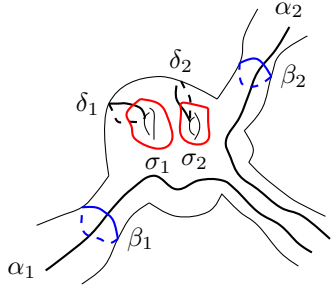
\begin{center}
\begin{tikzpicture}[scale=0.6]

\draw  plot[smooth, tension=.7] coordinates {(-2.6,0.1) (-2.6,1.7) (-1.5,2.3) (-0.2,1.8)};
\draw  plot[smooth, tension=.7] coordinates {(-0.2,1.8) (0.3,2.6) (0.4,3.2) (0.7,3.8)};
\draw  plot[smooth, tension=.7] coordinates {(0.4,1.1) (0.8,1.6) (1,2.3) (1.2,2.8) (1.8,3.2)};
\draw  plot[smooth, tension=.7] coordinates {(0.4,1.1) (0.6,0.7) (0.5,0.4)};
\draw  plot[smooth, tension=.7] coordinates {(0.5,0.4) (0.8,0.1) (1.3,-0.3) (1.6,-0.5) (2.2,-0.8)};
\draw  plot[smooth, tension=.7] coordinates {(1.6,-1.9) (1.2,-1.6) (0.7,-1) (0.3,-0.7) (-0.1,-0.3)};
\draw  plot[smooth, tension=.7] coordinates {(-0.1,-0.3) (-0.4,-0.7) (-1.1,-0.8) (-1.7,-0.5)};
\draw  plot[smooth, tension=.7] coordinates {(-1.7,-0.5) (-2.1,-0.8) (-2.4,-1.3) (-2.7,-1.9) (-3.2,-2.3)};
\draw  plot[smooth, tension=.7] coordinates {(-2.6,0.1) (-2.9,-0.3) (-3.2,-0.6) (-3.7,-0.8) (-4.2,-1)};

\draw[thick]  plot[smooth, tension=.7] coordinates {(-4,-2) (-3.6,-1.7) (-3,-1.2) (-2.5,-0.6) (-2.1,-0.2) (-1.6,0.1) (-1.1,-0.2) (-0.6,0) (-0.1,0) (0.2,-0.2) (0.7,-0.7) (1.2,-1) (1.6,-1.4) (1.9,-1.6)};
\node[above,left] at (-4,-2) {$\alpha_1$};
\draw[thick]  plot[smooth, tension=.7] coordinates {(2.3,-1.3) (1.7,-0.8) (1.2,-0.6) (0.7,-0.2) (0.3,0.1) (0.1,0.5) (0,0.9) (0.3,1.4) (0.5,1.9) (0.7,2.4) (1.1,2.9) (1.4,3.4)};
\node[above] at (1.4,3.4) {$\alpha_2$};
\draw[blue, thick]  plot[smooth, tension=.7] coordinates {(-2.4,-1.3) (-2.5,-0.9) (-2.7,-0.7) (-2.9,-0.6) (-3.2,-0.6)};
\draw[blue, dashed, thick]  plot[smooth, tension=.7] coordinates {(-2.4,-1.3) (-2.6,-1.4) (-2.8,-1.3) (-3.1,-1.2) (-3.2,-1) (-3.2,-0.8) (-3.2,-0.6)};
\node[right] at (-2.4,-1.3) {$\beta_1$};
\draw[blue, thick]  plot[smooth, tension=.7] coordinates {(0.3,2.6) (0.6,2.7) (0.8,2.6) (1,2.3)};
\draw[blue, dashed, thick]  plot[smooth, tension=.7] coordinates {(1,2.3) (0.7,2.1) (0.4,2.1) (0.3,2.3) (0.3,2.6)};
\node[right] at (1,2.3) {$\beta_2$};
\draw  plot[smooth, tension=.7] coordinates {(-1.7,1.6) (-1.6,1.5) (-1.6,1.2) (-1.6,0.9)};
\draw  plot[smooth, tension=.7] coordinates {(-1.7,1.5) (-1.8,1.4) (-1.8,1.2) (-1.7,1)};
\draw  plot[smooth, tension=.7] coordinates {(-0.7,1.5) (-0.6,1.3) (-0.6,1.1) (-0.7,0.9)};
\draw  plot[smooth, tension=.7] coordinates {(-0.7,1.4) (-0.8,1.3) (-0.8,1.1) (-0.7,1)};

\draw[thick]  plot[smooth, tension=.7] coordinates {(-1.8,1.4) (-1.9,1.6) (-2.2,1.7) (-2.4,1.7) (-2.6,1.7)};
\draw[dashed, thick]  plot[smooth, tension=.7] coordinates {(-2.6,1.7) (-2.5,1.5) (-2.4,1.3) (-2.2,1.3) (-1.9,1.3) (-1.8,1.4)};
\node[left] at (-2.6,1.7) {$\delta_1$};
\draw[thick]  plot[smooth, tension=.7] coordinates {(-0.8,1.3) (-0.9,1.5) (-1.1,1.8) (-1.1,2.1)};
\draw[dashed, thick]  plot[smooth, tension=.7] coordinates {(-1,2.2) (-0.9,2) (-0.8,1.8) (-0.8,1.5) (-0.8,1.3)};
\node[above] at (-1,2.2) {$\delta_2$};
\draw[red, thick]  plot[smooth, tension=.7] coordinates {(-1.9,1.8) (-2.2,1.5) (-2.1,1.2) (-2,0.9) (-1.5,0.7) (-1.2,0.9) (-1.3,1.4) (-1.5,1.7) (-1.9,1.8)};
\node[below] at (-1.5,0.7) {$\sigma_1$};
\draw[red, thick]  plot[smooth, tension=.7] coordinates {(-0.4,1.6) (-0.8,1.7) (-1,1.7) (-1,1.5) (-1,1) (-0.7,0.8) (-0.4,0.9) (-0.4,1.3) (-0.4,1.6)};
\node[below] at (-0.7,0.8) {$\sigma_2$};
\end{tikzpicture}
\end{center}

We also define $\nu^*_i=\whodge\nu_i$ for $i=1,\ldots,g$, which remains $\wltwo$-integrable because each $\nu_i$ has compact support.
Now have the set of closed $1$-forms $\{\nu_1,\ldots,\nu_g, \nu^*_1,\ldots,\nu^*_g,\eta_1,\ldots,\eta_{r-1}\}\subset\wltwoforms^1(\Sigma)$.

We will now show that each of these is cohomologous to a smooth Gaussian Harmonic $1$-form.
Let us follow $\cite{mcgonagle2015}$ (see also \cite{bueler1999}) and define $A\subset\wltwoforms^1(\Sigma)$ as the closure of $\{df:f\in\cptf(\Sigma)\}$, the space of exact $1$-forms, $B\subset\wltwoforms^1(\Sigma)$ as the closure of $\{\hodged\xi:\xi\in\Omega^2_c(\Sigma)\}$, the space of $\hodged$-exact $1$-forms and $H\subset\wltwoforms^1(\Sigma)$ as the orthogonal complement of $A\oplus B$, that is, the space of weakly Gaussian Harmonic forms.
We have the orthogonal decomposition $\wltwoforms^1(\Sigma)=A\oplus B\oplus H$.

Since $\nu_i$ and $\eta_k$ are closed, then $\nu_i,\eta_k\in A\oplus H$ for all $i=1,\ldots,g$ and $k=1,\ldots,r-1$.
Similarly, $\gauss\star\nu^*_i\in A\oplus H$ for all $i=1,\ldots,g$.

Finally, denote by $\omega_i$, $\omega_i^*$ and $\theta_k$ the projection of $\nu_i$, $\nu_i^*$ and $\eta_k$ onto $H$.
We will show that these forms are non-zero, linearly independent and smooth.

\begin{claim} Each $\omega_i$, $\omega_i^*$ and $\theta_k$ is non-zero.
\end{claim}
First, observe that for any $f\in\cptf(\Sigma)$ we have
\begin{equation*}
\int_\Sigma\tau_j\wedge(\nu_i+df)=\int_{\sigma_j}(\nu_i+df)=\delta_{ij}.
\end{equation*}
Since $\nu_i-\omega_i\in A$, that is, $\nu_i-\omega_i=\lim d f$. By taking the limit on the left hand side we can conclude that $\int_\Sigma\tau_j\wedge\omega_i=\delta_{ij}$.

Second, observe $\nu_i-\gauss\star\omega_i^*\in A$.
Similarly to the above we may show that $\int_\Sigma\gauss\star\tau_j\wedge\omega_i^*=\delta_{ij}$.

Again, note that $\eta_k-\theta_k\in A$ and
\begin{equation*}
\int_\Sigma\gamma_m\wedge(\eta_k+df)=\int_{\beta_m}(\eta_k+df)=\delta_{mk}
\end{equation*}
for any $f\in\cptf(\Sigma)$.
By taking the limit we conclude that $\int_\Sigma\gamma_m\wedge\theta_k=\delta_{mk}$, which proves the claim.

\begin{claim} The $1$-forms $\omega_i$, $\omega_i^*$ and $\theta_k$ are all linearly independent.
\end{claim}
It is a straightforward application of the divergence theorem to show that $\nu_i+df$ and $\nu_j^*+\hodged\xi$ are pairwise orthogonal in $\wltwoforms^1(\Sigma)$ for any $f$ and $\xi$.
We simply observe that $\nu_i$, $\nu_i^*$ are always orthogonal and $\nu_i$, $\nu_j^*$  have disjoint compact support when $i\neq j$.
It then follows by taking the limit that $\omega_i$ and $\omega_i^*$ are also orthogonal.

Now, following the same computations as in the previous claim we can show that $\int_\Sigma\gamma_k\wedge\omega_i=\int_\Sigma\gamma_k\wedge\omega_i^*=0$ for all $i,k$.
Recall that $\gamma_k$ is supported in a neighbourhood of $\beta_k$ and $\nu_i,\nu_i^*$ are supported in a neighbourhood of $\sigma_i$, which are disjoint curves.

Suppose there exists $a_i, a_i^*,b_k$ such that
\begin{equation*}
\sum_{i=1}^g a_i\omega_i+a_i^*\omega_i^*+\sum_{k=1}^{r-1}b_k\theta_k=0.
\end{equation*}
By taking the integral $\int_\Sigma\gamma_k\wedge\cdot$ we have that $b_k=0$ for all $k=1,\ldots,r-1$.
That is, $\sum_{i=1}^g a_i\omega_i+a_i^*\omega_i^*=0$.
However, $\omega_i,\omega_j^*$ are pairwise orthogonal, so $a_i=a_i^*=0$.
This proves the claim.

It remains to show that $\omega_i$, $\omega_i^*$ and $\theta_k$ are all smooth $1$-forms.
We simply observe that, in local coordinates they satisfy a weakly linear elliptic system of PDE's so regularity follows from classical theory (see for example \cite{douglis-nirenberg1955}).

This concludes the proof since $\{\omega_i,\omega_i^*,\theta_k\}\subset\ghf^1(\Sigma)$ is a set of $2g+r-1$ linearly independent $\wltwo$-integrable GHF's.
\end{proof}

\begin{remark}
It is not clear how to improve the dimension to $2(g+r-1)$ as in the minimal surface case \cite{chodosh-maximo2016} since $\whodge\eta_i$ are no longer $\wltwo$-integrable as constructed above.
\end{remark}

\section{Index Estimates}
In this section we use the calculations by Mcgonagle in \cite{mcgonagle2015} together with our previous results to improve the index estimates of self-shrinkers when the ends are asymptotically conical.
Given a $1$-form $\omega$ in $\Sigma$ and a vector $v\in\real^3$ we will denote by $\omega^\#$ the unique vector field defined by $\omega(X)=g_\Sigma(\omega^\#,X)$ and $v^\flat$ the $1$-form on $\real^3$ defined by $v^\flat(X)=\langle v,X\rangle$.

The following computation is carried out in \cite{mcgonagle2015} and we refer the reader to the original article.

\begin{lemma}[{\cite[Corollary 1.1, Lemma 2.2]{mcgonagle2015}}]\label{lemma computation}
Let $\omega$ be a Gaussian Harmonic $1$-form on a self-shrinking surface $\Sigma^2\subset\real^3$, $v\in\real^3$ a fixed vector.
Then, 
\begin{equation*}
\wslaplace\omega=\frac{\omega}{2}-A^2(\omega^\#,\cdot),
\end{equation*}
where $A^2=\langle\rnabla_\cdot N_\Sigma, \rnabla_\cdot N_\Sigma\rangle$.
And,
\begin{equation*}
\wslaplace\langle \omega, v^\flat\rangle=\frac{\langle\omega,v^\flat\rangle}{2}-2A^2(\omega^\#,v^\top)-2\langle N_\Sigma,v \rangle\langle\snabla\omega,A\rangle.
\end{equation*}
\end{lemma}


The next result is an extension of Theorem $4.1$ by Mcgonagle on \cite{mcgonagle2015}.
The arguments are exactly the same except for the fact that we are able to construct more GHF's.
We include the proof for completeness.

\begin{theorem}\label{indexestimate}
Let $\Sigma$ be a properly immersed orientable self-shrinking surface in $\real^3$ with genus $g$ and $r$ asymptotically conical ends.
If the principal curvatures satisfy $|\kappa_1^2-\kappa_2^2|\leq c <1$ then, 
\begin{equation*}
\index(\Sigma)\geq\frac{2g+r-1}{3}.
\end{equation*}
\end{theorem}
\begin{proof}
We may assume that $\index(\Sigma)=I$ is finite so Proposition \ref{fischercolbrie} applies to produce $f_1,\ldots,f_I\in\wltwo(\Sigma)$ such that $Q^\Sigma(u,u)\geq 0$ whenever $u\in\cptf(\Sigma)$ is $\wltwo$-orthogonal to $f_i$, for all $i=1,\ldots,I$.

It follows from Theorem \ref{dimension} that there exists a $(2g+r-1)$-dimensional space $V\subset\wltwoforms^1\cap\ghf^1$ of Gaussian Harmonic $1$-forms.
Now, let $\phi$ be a cut-off function in $\real^n$ such that $\phi=1$ on $B_R$, $\phi=0$ on $\real^n\setminus B_{2R}$ and $|\rnabla\phi|^2\leq\frac{2}{R^2}$.
If $R$ is sufficiently large, have that $dim(V_\phi)=dim(V)$, where $V_\phi=\{\phi\omega:\omega\in V\}$.
We want to show that $3I\geq dim(V_\phi)$

Let $F:V_\phi\rightarrow\real^{3I}$ be given by $F(\phi\omega)=(\int_\Sigma \phi\omega^j f_i)_{i=1,\ldots,I}^{j=1,2,3}$, where $\omega^j=\langle\omega,dx^j\rangle$ are the coordinates of $\omega$ in $\real^3$.
By contradiction suppose that $3I<dim(V_\phi)$, then there exists $\omega\in V$ such that $F(\phi\omega)=0$.
That is, $\phi\omega^j$ is $\wltwo$-orthogonal to $f_i$ for all $i=1,\ldots,I$ and $j=1,2,3$ so, 
\begin{equation*}
-\int_\Sigma\phi\omega^j\wsjacobi(\phi\omega^j)\dgauss\geq 0.
\end{equation*}

On the other hand, since $\omega$ is a GHF, we use Lemma \ref{lemma computation} for each $j=1,2,3$ to compute:
\begin{equation*}
\begin{aligned}
   -\int_\Sigma\phi\omega^j\wsjacobi&(\phi\omega^j)\dgauss  = -\int_\Sigma\phi^2\omega^j\wsjacobi\omega^j\dgauss + \int_\Sigma|\snabla\phi|^2(\omega^j)^2\dgauss\\
                                                          & = -\int_\Sigma\phi^2(\omega^j)^2[1+(\kappa_1^2+\kappa_2^2)]\dgauss \\ & +\int_\Sigma [2A^2(\omega^\#,\omega^je_j^\top)+2\langle N_\Sigma,\omega^je_j \rangle\langle\snabla\omega,A\rangle+|\snabla\phi|^2(\omega^j)^2]\dgauss.
\end{aligned}
\end{equation*}
Adding it for $j=1,2,3$ we obtain
\begin{equation*}
\begin{aligned}
-\sum_{j=1}^3\int_\Sigma\phi\omega^j\wsjacobi(\phi\omega^j)\dgauss & = \int_\Sigma|\snabla\phi|^2|\omega|^2\dgauss-\int_\Sigma\phi^2|\omega|^2+\sum_{i\neq j}(\kappa_j^2-\kappa_i^2)(\omega^i)^2\dgauss\\
                                                                   & \leq\int_\Sigma|\snabla\phi|^2|\omega|^2\dgauss-\int_\Sigma\phi^2|\omega|^2(1-|\kappa_1^2-\kappa_2^2|)\dgauss\\
                                                                   & \leq \int_\Sigma|\snabla\phi|^2|\omega|^2\dgauss-(1-c)\int_\Sigma\phi^2|\omega|^2\dgauss\\
                                                                   & \leq \left(\frac{2}{R^2}-(1-c)\right)\int_\Sigma|\omega|^2\dgauss
\end{aligned}
\end{equation*}
By taking $R>0$ sufficiently large we conclude that $-\sum_{j=1}^3\int_\Sigma\phi\omega^j\wsjacobi(\phi\omega^j)\dgauss<0$ which is a contradiction.
We conclude that $3I\geq dim(V_\phi)=2g+r-1$.
\end{proof}

We would like to conclude this section by improving the above thanks to the recently published result by Impera-Rimoldi-Savo \cite{impera-rimoldi-savo2018:arxiv}.
The following is a direct application of \cite[Theorem C]{impera-rimoldi-savo2018:arxiv} and Theorem \ref{dimension} in the previous section.

\begin{corollary}
Let $\Sigma$ be a properly immersed orientable self-shrinking surface in $\real^3$ with genus $g$ and $r$ asymptotically conical ends.
Then,
\begin{equation*}
\index(\Sigma)\geq\frac{2g+r-1}{3}+1.
\end{equation*}
\end{corollary}

\bibliographystyle{plain}
\bibliography{bibliography}

\end{document}